\documentclass[12pt,a4paper, twoside]{article}
\usepackage[utf8]{inputenc}
\usepackage{amsmath}
\usepackage{amsfonts}
\usepackage{amssymb}
\usepackage{graphicx}
\usepackage{enumitem}
\author{Marco Bertenghi}
\usepackage{amsthm}
\usepackage{bbm}
\usepackage{xcolor}
\usepackage{centernot}
\usepackage{doi}
\usepackage{booktabs}
\usepackage{afterpage}
\date{}
\usepackage[centering, a4paper]{geometry}

\newtheorem{thm}{Theorem}[section]

\newtheorem{cor}{Corollary}[section]

\theoremstyle{definition}

\usepackage[square, numbers]{natbib}
  
\newcommand{\EE}{\mathbb{E}} 
\newcommand{\PP}{\mathbb{P}} 

\usepackage{hyperref}
\hypersetup{
	citecolor=blue,
    colorlinks=true,
    linkcolor=blue,
    filecolor=magenta,      
    urlcolor=cyan,
}

\title{Functional limit theorems for the Multi-dimensional Elephant Random Walk}
\author{Marco Bertenghi\footnote{Institute of Mathematics, University of Zurich, Winterthurerstrasse 190, CH-8057 Zürich, Switzerland}}
\date{\today}

\begin{document}
\maketitle
\begin{abstract}
In this article we shall derive functional limit theorems for the multi-dimensional elephant random walk (MERW) and thus extend the results provided for the one-dimensional marginal by Bercu and Laulin in \cite{berculaulin}. The MERW is a non-Markovian discrete-time random walk on $\mathbb{Z}^d$ which has a complete memory of its whole past, in allusion to the traditional saying that \textit{an elephant never forgets}. As the name suggests, the MERW is a $d$-dimensional generalisation of the elephant random walk (ERW), the latter was first introduced by Schütz and Trimper \cite{schuetz} in 2004. We measure the influence of the elephant's memory by a so-called \textit{memory parameter} $p$ between zero and one. A striking
feature that has been observed in \cite{schuetz} is that the long-time behaviour of the ERW exhibits a phase transition at some critical memory parameter $p_c$. We investigate the asymptotic behaviour of the MERW in all memory regimes by exploiting a connection between the MERW and Pólya urns, following similar ideas as in the work by Baur and Bertoin in \cite{bertoinbaur} for the ERW.
\end{abstract}
\section{Introduction}
The elephant random walk (ERW), first introduced by Schütz and Trimper in 2004, is an interesting simple random process arising from mathematical physics \cite{schuetz}. The ERW
model is a discrete-time nearest-neighbour
random walk on the integers, which remembers its full past and
chooses its next step according to the following dynamics: First, it selects uniformly at random a step from the past, and then, with probability $p \in (0,1)$
it repeats what it did at the remembered time, whereas with complementary probability $1 - p$, it makes a step in the opposite direction. Notably, in the borderline case of $p=1/2$, the memory has no longer an effect on the elephant's movement and we obtain the classical simple random walk on the integers. These heuristics of the ERW naturally translate into higher dimensions: We consider the $d$-dimensional grid $\mathbb{Z}^d$ for some $d \geq 1$, then, the elephant selects uniformly at random a step from its past and, with probability $p$, it repeats what it did at the remembered time, whereas with probability $(1-p)/(2d-1)$, it moves from its current position to one of the remaining $2d-1$ neighbours. We refer to the next section for a precise definition.

Over the last decade many references addressing the ERW have emerged; see \cite{arxivbaur, bertoinbaur, bercu, berculaulin, colettischuetz,  kuersten} with further references therein. A striking feature that has been pointed at in those works is that the long-time behaviour of the ERW exhibits a phase transition at some critical memory parameter $p_c$. Governed by this critical memory parameter $p_c$, it is common in the literature to distinguish between three regimes: the so-called \textit{diffusive regime} ($p<p_c)$, the \textit{critical regime} ($p=p_c$) and the \textit{superdiffusive regime} ($p>p_c$). In contrast to the ERW the MERW, while being a natural process to consider, has received far less attention than its one-dimensional special case. In \cite{cressoni} and \cite{lyu}  the authors consider the special case of $d=2$. Quite recently, Bercu and Laulin \cite{berculaulin} established in the
\textit{diffusive regime} almost sure convergence in the sense of the law of large numbers for the MERW. Further, they provided the asymptotic normality of the
MERW when properly normalised in both, the \textit{diffusive} and \textit{critical regimes}. In the \textit{superdiffusive regime}, they proved the
almost sure convergence of the MERW to some non-degenerate random variable. Their
analysis relied on asymptotic results for multi-dimensional martingales.

 Here, as in the article of Baur and Bertoin \cite{bertoinbaur}, we use a connection between the MERW and Pólya urns. Said connection allows us to establish functional limit theorems for the MERW by using the results provided by Janson \cite{janson} for the more general framework of Pólya urns. Notably, our limit theorems are stronger than finite-dimensional
convergence of the MERW. They imply convergence
of continuous functionals of the properly rescaled MERW  and thus generalise results shown in Bercu and Laulin \cite{berculaulin}. For  a direct application, we show that the (rescaled) Euclidean norm process of the MERW converges to a Bessel process and we give a direct proof of a recent result by Bercu and Laulin \cite{berculaulincenter} on the center of mass of the MERW.

We stress that the ERW is a \textit{time-inhomogeneous} Markov chain, although some works in the literature improperly assert its non-Markovian character. In contrast, the MERW in dimensions greater or equal to two is indeed non-Markovian. 

The rest of this paper is arranged as follows: In the next section, we give a precise mathematical definition of the MERW. In Section \ref{sec:three}, we introduce a particular discrete-time urn model which will be used to describe the position of the MERW and the connection between the two processes becomes apparent. Finally, we show in Section \ref{sec:four} how the known limit results on the composition of the urn process provided by Janson in \cite{janson} can be translated into statements
about the position of the MERW when time goes to infinity. For the reader's convenience, we gather rather technical calculations used in Section \ref{sec:four} in an Appendix \ref{Appendix}. 
\section{Definition of the MERW}
We now introduce the exact definition of the MERW. The MERW is a $d$-dimensional nearest-neighbour random walk $(S_n)_{n \in \mathbb{N}_0}$ on the grid $\mathbb{Z}^d$. We agree that the MERW is started from the origin in $\mathbb{Z}^d$ at time zero. For any time $n \geq 1$ the position $S_n$ of the MERW can be described as follows: At initial time $n=1$, for the sake of definiteness, the elephant takes a deterministic step in the direction of $e_1=(1,0, \dots , 0) \in \mathbb{R}^d$. Afterwards, at any time $n \geq 2$, the elephant chooses uniformly at random  a number $n'$ among the previous times $1, \dots ,n-1$ and, independently of its path up to time $n-1$. Then, with probability $p \in (0,1)$ the elephant moves in the same direction as it did for the time $n'$, or else, it moves in one of the remaining $2d-1$ directions with the same probability of $(1-p)/(2d-1)$ for each remaining direction.

 More precisely: Let $e_1, \dots , e_d$ denote the standard basis for the Euclidean space $\mathbb{R}^d$ and let $S_0= \mathbf{0},$ where we shall denote by $\mathbf{0}:= (0, \dots , 0) \in \mathbb{Z}^d$ the origin of $\mathbb{Z}^d$. At any time $n \geq 1$, the position of the MERW is given by 
 \begin{align*}
 S_n = S_{n-1}+ \sigma_n,
 \end{align*}
 where $\sigma_n,$ $n \in \mathbb{N}=\{1,2, \dots\}$ are random variables taking values in $\{\pm e_1, \dots , \pm e_d\}$, which are specified as follows: First, we set $\sigma_1=e_1$ and we note that all the remaining $2d-1$ directions can be obtained by rotations. At any later time $n \geq 2$,  we choose a number $n'$ uniformly at random among the previous times $1, \dots , n-1$ and set $M_{n'}:= \{ \pm e_1, \dots , \pm e_d\} \setminus \{\sigma_{n'}\}$. We then define
 \begin{align*}
\PP( \sigma_n = \sigma_{n'})&=p, \\ 
\PP( \sigma_n = \sigma) &= \frac{1-p}{2d-1} \qquad \text{for all } \sigma \in M_{n'},
\end{align*}
where $p \in (0,1)$ is a memory parameter which is inherent to the model. We implicitly agree that the various random choices made in this construction are independent from each other. Notably, when $d=1$, we obtain again the ERW as described by Schütz and Trimper in \cite{schuetz}, see also \cite[Section 2]{bertoinbaur}.

\section{Connection to Pólya-type urns} \label{sec:three}
In this section we explain how to model the MERW with the help of Pólya-type urns.

 Imagine a discrete-time urn with balls of $2d$ distinct colours. The composition of the urn at any time $n \in \mathbb{N}$ is specified by a $2d$-dimensional random vector $X_n=(X_n^1,X_n^2, \dots ,X_n^{2d})$, where each component $X_n^i$ counts the number of balls of colour $i \in \{1, \dots ,2d\}$ present in the urn at time $n$. We agree that $X_0=\mathbf{0} \in \mathbb{Z}^{2d}$, that is there are no balls of any colour present in the urn at time zero. Further, we restrict ourselves to the initial deterministic configuration $X_1 = (1,0, \dots , 0) \in \mathbb{Z}^{2d}$ meaning that we start the urn with exactly one ball of a certain colour. The dynamics of the urn are then given as follows: At any time $n \geq 2$ , we draw a ball uniformly at random from the urn, observe its colour, put it back to the urn and add with probability $p$ a ball of the same colour, or add a ball of the $2d-1$ remaining colour each with probability $(1-p)/(2d-1)$. Finally, we update $X_n=(X_n^1, \dots ,  X_n^{2d})$ accordingly, such that $X_n$ describes the composition of the urn after the $(n-1)$\textit{st} drawings. In light of this construction, we notice that the urn contains colour-coded information, in form of balls, of the steps taken by the MERW up until time $n$. 

The connection to the MERW model is remarkably simple: If $S_n$ denotes the position of the MERW started from zero at time zero, then
\begin{align} \label{eq:connectionMERWUrn}
(S_n)_{n \in \mathbb{N}} =_d \left( \left(X_n^1-X_n^2 \right)e_1 + \dots + \left( X_n^{2d-1}-X_n^{2d} \right) e_d \right)_{n \in \mathbb{N}}
\end{align}
where $=_d$ refers to equality in law. Evidently we obtain again that $S_1=e_1$ thanks to our initial configuration of $X_1$. Observe that (\ref{eq:connectionMERWUrn}) is a representation of the process $(S_n)_{n \in \mathbb{N}}$ as a continuous function of the urn model $(X_n)_{n \in \mathbb{N}}$.  In Section \ref{sec:four}, where we state and prove our main results, we shall crucially exploit (\ref{eq:connectionMERWUrn}) in conjunction with the continuous mapping theorem. 

The urn described to model the MERW fits into a broader setting of so-called generalised Friedman's or Pólya urns. We refer to Janson \cite{janson} for a very broad treatment on this subject. 

Among others, important key quantities that govern the long-time behaviour of the urn process are the eigenvalues and eigenvectors of the so-called \textit{mean replacement matrix}, see Display (2.1) in Janson \cite{janson}. In our case, it is given by the $2d \times 2d$ matrix
\begin{align} \label{mat:meanreplacementmat}
A = \begin{pmatrix}
p & \frac{1-p}{2d-1} & \dots & \dots & \frac{1-p}{2d-1}\\
\frac{1-p}{2d-1} & \ddots & \ddots &  & \vdots \\ 
\vdots & \ddots&\ddots & \ddots & \vdots \\
 \vdots &  & \ddots  & \ddots & \frac{1-p}{2d-1} \\
 \frac{1-p}{2d-1} & \dots & \dots&  \frac{1-p}{2d-1} & p
\end{pmatrix}= \frac{1-p}{2d-1}J_{2d} + \frac{2dp-1}{2d-1}I_{2d},
\end{align}
where $J_{2d}$ denotes the matrix of ones in dimensions $2d \times 2d$ and $I_{2d}$ is the $2d\times 2d$ identity matrix. Elementary linear algebra yields that $A$ has the eigenvalues $\lambda_1=1$ of multiplicity one and $\lambda_2 = (2dp-1)/(2d-1)$ of multiplicity $2d-1$. The right and left eigenvectors corresponding to the largest eigenvalue $\lambda_1$ are given by $v_1 = \frac{1}{2d}(1, \dots ,1 )' \in \mathbb{R}^{2d}$ and $u_1=(1, \dots , 1) \in \mathbb{R}^{2d}$ respectively. 

It is classical, (consult for example, Athreya and Karlin \cite{karlin}, Janson \cite{janson}, Kesten and Stigum \cite{kesten}, Chauvin, Pouyanne and Sahnoun \cite{chauvin}) that the asymptotics of the urn depends on the position of $\lambda_2/ \lambda_1$ with respect to $1/2$, or, in the situation of a more general urn, assuming that the largest eigenvalue $\lambda^*$ is positive and simple, one has to check whether there exists an eigenvalue different from $\lambda^*$ such that its real part is strictly larger than $\lambda^*/2$. In our case, solving $\lambda_2/\lambda_1 =1/2$ shows, on a formal level, why for the MERW model a phase transition from a \textit{diffusive} to a \textit{superdiffusive regime} occurs at critical memory parameter \begin{align} \label{eq:criticalmemoryparameter}
p_c=p_c^d= \frac{2d+1}{4d}.
\end{align}
We observe that we rediscover the critical memory parameter displayed in (2.6) in Bercu and Laulin \cite{berculaulin}, moreover, for $d=1$ we arrive again at $p_c=3/4$, that is the phase transition observed by Schütz and Trimper \cite{schuetz}. According to this observation we shall refer to the memory parameter $p \in (0, p_c^d)$ as being in the \textit{diffusive regime}, for $p=p_c^d$ as being in the \textit{critical regime} and for $p \in (p_c^d,1)$ as being in the \textit{superdiffusive regime}.
\section{Results and proofs for the MERW} \label{sec:four}
In this section we shall derive functional limit theorems for the MERW in all regimes. For our purpose, it shows to be most fruitful to adapt the general results from Janson in \cite{janson} and translate them back into the setting of the MERW model by using the connection displayed in (\ref{eq:connectionMERWUrn}) and the continuous mapping theorem. Throughout this section we shall denote by $S_n$ the position of the MERW at time $n$ and by $X_n=(X_n^1, \dots , X_n^{2d})$ the configuration of the urn simulating the MERW at time $n$ as introduced in the previous section. 
\subsection{A LLN-type convergence result for the MERW in all regimes ($0<p<1$)}
Our first result deals with a strong law of large numbers for the MERW. In 2017 Coletti, Gava and Schütz \cite{colettischuetz} established a LLN-type of convergence in all regimes $p \in (0,1)$ for the ERW, see Theorem 1 and Corollary 1 in \cite{colettischuetz}. Later, in 2019, Bercu and Laulin extended this result, see Theorem 3.1 in \cite{berculaulin}, for the MERW in its diffusive regime $0<p<p_c^d$. Observe that their LLN for the MERW crucially depends on $p_c^d$ and thus on the dimension $d$. Here, we shall extend the findings of Bercu and Laulin to $p \in (0,1)$ and thus eliminate the dependency on the dimension. In light of the result provided by Coletti, Gava and Schütz, this then establishes the anticipated behaviour also in higher dimensions for the MERW.
\begin{thm}[Strong law of large numbers for the MERW] For all $p \in (0,1)$ we have the almost sure convergence as $n$ tends to infinity
\begin{align*}
\frac{S_n}{n} \to \mathbf{0}.
\end{align*}
\end{thm}
\begin{proof}
By Theorem 3.21 in  Janson \cite{janson} (see also Section V.9.3. in \cite{athreya}) we have the almost sure convergence as $n$ tends to infinity of $n^{-1} X_n \to  \lambda_1 v_1$, where $\lambda_1$ is the largest eigenvalue of the mean replacement matrix $A$ as given in (\ref{mat:meanreplacementmat}) and $v_1$ is its corresponding (right) eigenvector. As $\lambda_1 =1$ and $v_1= \frac{1}{2d}(1, \dots , 1)' \in \mathbb{R}^{2d}$ the claim follows by the continuous mapping theorem and the representation established in (\ref{eq:connectionMERWUrn}). This concludes the proof.
\end{proof}
\subsection{The diffusive regime $(0<p< p_c^d)$}
In the diffusive regime, we provide a functional central limit theorem for the MERW which will generalise the asymptotic normality provided by Bercu and Laulin, see Theorem 3.3 in \cite{berculaulin}. This convergence result deals with a distributional convergence of the MERW, which holds in the Skorokhod space $D([0, \infty))$ of right-continuous functions with left-hand limits. We refer to Chapter VI in \cite{Jacod} for background on the Skorokhod topology and weak convergence of stochastic processes.
\begin{thm}[Central limit theorem for the MERW]\label{thm:cltMERW} Let $0<p<p_c^d$. Then, for $n$ tending to infinity, we have the distributional convergence in $D([0, \infty))$:
\begin{align*}
\left( \frac{S_{\lfloor tn \rfloor}}{\sqrt{n}}, t \geq 0 \right) \implies (W_t, t \geq 0),
\end{align*}
where $W=(W_t, t \geq 0)$ is a continuous $\mathbb{R}^d$-valued centered Gaussian process started from zero and with covariance structure specified by 
\begin{align} \label{disp:covariancestructurediffusive}
\EE(W_sW_t')= \frac{(2d-1)}{d(1+2d-4dp)}s \left( \frac{t}{s}\right)^{ \frac{2dp-1}{2d-1}} I_d, \quad 0 < s \leq t. 
\end{align}
\end{thm}
For $t=s=1$, we rediscover the asymptotic normality established under Theorem 3.3 by Bercu and Laulin \cite{berculaulin}. Furthermore, specialising on $d=1$ yields Theorem 1 established by Baur and Bertoin \cite{bertoinbaur}.
\begin{proof}[Proof of Theorem \ref{thm:cltMERW}]
We apply Theorem 3.31(i) of Janson \cite{janson}, which establishes that 
\begin{align*}
\left( n^{-1/2} \left( X_{\lfloor tn\rfloor} - tn \lambda_1 v_1 \right), t \geq 0 \right)
\end{align*}
converges in distribution in $D([0, \infty))$ towards a continuous $\mathbb{R}^d$-valued centered Gaussian process $V=(V_t, t \geq 0)$ started from zero. In our case, we have $\lambda_1=1$ and $v_1 = \frac{1}{2d}(1, \dots , 1) \in \mathbb{R}^{2d}$. Further, the covariance structure of $V$ is closer specified under Remark 5.7 Display (5.6) in Janson \cite{janson}. In light of said remark we obtain 
\begin{align*}
\EE(V_sV_t')= s \Sigma_I e^{ \log(t/s)A}, \quad 0 < s \leq t,
\end{align*}
with $\Sigma_I$ being a $2d \times 2d$ matrix specified under (2.15) in Janson \cite{janson}. A rather technical calculation yields (see Appendix \ref{Appendix} for details)
\begin{align} \label{eq:sigmaDiffusive}
\Sigma_I = \frac{2d-1}{4d^2(1+2d-4dp)} \begin{pmatrix}
2d-1 & -1 & \dots & -1  \\
-1 & \ddots  & \ddots  & \vdots  \\
\vdots  & \ddots & \ddots & -1  \\
-1  & \dots & -1 & 2d-1
\end{pmatrix}.
\end{align}
The computation of the matrix exponential is elementary (see Appendix \ref{Appendix} for more details)
 and together with (\ref{eq:sigmaDiffusive}) we obtain for $0 < s \leq t$, 
\begin{align} \label{eq:covariancealmostdone}
\EE(V_sV_t')  = \frac{ (2d-1)}{4d^2(1+2d-4dp)}s \left( \frac{t}{s}\right)^{ \lambda_2} \begin{pmatrix}
2d-1 & -1 & \dots & -1 \\
-1 & \ddots & \ddots & \vdots \\
\vdots & \ddots  & \ddots & -1  \\
-1 & \dots & -1 & 2d-1
\end{pmatrix},
\end{align}
where $ \lambda_2=(2dp-1)/(2d-1)$. With an appeal to (\ref{eq:connectionMERWUrn}) we deduce that 
\begin{align*}
\left( n^{-1/2} \left( S_{\lfloor tn \rfloor} \right), t \geq 0 \right)
\end{align*} 
converges in distribution in $D([0,  \infty))$ towards a $\mathbb{R}^d$-valued process, denoted by $W=(W_t, t \geq 0)$ and specified by $W_t= (V_t^1-V_t^2)e_1' + \dots + (V_t^{2d-1}-V_t^{2d})e_d'$ almost surely, where $V^i$ denotes the $i$-th component of the process $V$ for all $i=1, \dots , 2d$. Since $V$ is a continuous $\mathbb{R}^{2d}$-valued centered Gaussian process started from zero, and $W$ is a linear combination of the components of $V$, it follows that $W$ is also a continuous $\mathbb{R}^d$-valued centered Gaussian process started from zero. With an appeal to the covariance structure (\ref{eq:covariancealmostdone}) of $V$ one verifies that the covariance structure of $W$ is given by (\ref{disp:covariancestructurediffusive}). This proves the assertion.
\end{proof}
We now discuss two consequences of Theorem \ref{thm:cltMERW}.

We introduce for $i=1, \dots , d$ the processes
\begin{align} \label{process:ReinforcedBM}
W_i(t) = \sqrt{ \frac{2d-1}{d(1+2d-4dp)} }t^{\frac{2dp-1}{2d-1}} B_i \left( t^{ \frac{1+2d-4dp}{2d-1}} \right), \qquad t \in \mathbb{R},
\end{align}
where $B_1, \dots , B_d$ are independent one dimensional Brownian motions. In \cite{BertoinUniversality}, for the special case of $d=1$, the process displayed in (\ref{process:ReinforcedBM}) is referred to as a \textit{noise reinforced Brownian motion} and   belongs to a larger class of reinforced processes recently introduced by Bertoin in \cite{BertoinNoise} called \textit{noise reinforced Lévy processes}. It is straightforward to check that $W=(W_1, \dots , W_d)$ is a Gaussian process with the covariance given in (\ref{disp:covariancestructurediffusive}). In other words, the coordinates of $W$ are independent \textit{noise-reinforced Brownian motions}. This insight allows us to prove the convergence of the (rescaled) Euclidean norm process towards a properly scaled Bessel process.

\begin{cor} Let $0<p<p_c^d$. Then, for $n$ tending to infinity, we have the distributional convergence in the sense of Skorokhod
\begin{align*}
\left( \frac{ \| S( \lfloor nt \rfloor ) \|}{\sqrt{n}}, t \geq 0 \right) \implies \left( \sqrt{ \frac{2d-1}{d(1+2d-4dp)} }t^{\frac{2dp-1}{2d-1}} R\left( t^{ \frac{1+2d-4dp}{2d-1}}\right),t \geq 0 \right),
\end{align*}
where $\| \cdot \|$ denotes the Euclidean norm in $\mathbb{R}^d$ and $R=(R(t), t \geq 0)$ is a $d$-dimensional Bessel process started from the origin.
\end{cor}
\begin{proof}
Evidently $n^{-1/2} \| S ( \lfloor nt \rfloor ) \|$ is a continuous functional of $n^{-1/2} S_{\lfloor nt \rfloor}$ in the Skorokhod space $D([0, \infty))$. Hence we can apply Theorem \ref{thm:cltMERW} and obtain as $n$ tends to infinity
\begin{align*}
\frac{\| S ( \lfloor nt \rfloor) \|}{\sqrt{n}} = \left \| \frac{S ( \lfloor nt \rfloor)}{\sqrt{n}} \right\| \implies  \| W_t\| = \sqrt{ \frac{2d-1}{d(1+2d-4dp)} }t^{\frac{2dp-1}{2d-1}} \left\|B \left(t^{ \frac{1+2d-4dp}{2d-1}} \right)\right\|,
\end{align*}
where $B=(B_1(t), \dots , B_d(t), t \geq 0 )$ denotes a $d$-dimensional Brownian motion. 
\end{proof}
Quite recently, Bercu and Laulin in \cite{berculaulincenter} studied the center of mass 
\begin{align*}
G_n = \frac{1}{n}\sum_{k=1}^n S_k
\end{align*}
associated to the MERW. The asymptotic normality of the center of mass of the MERW (CMERW) is then an immediate consequence of the central limit theorem for MERW, Theorem \ref{thm:cltMERW}. 
\begin{cor}[Bercu and Laulin \cite{berculaulincenter}, Theorem 2.3] Let $0<p<p_c^d$ and let us denote $a=(2dp-1)/(2d-1)$. Then, for $n$ tending to infinity, we have the convergence in distribution 
\begin{align*}
\frac{1}{\sqrt{n}}G_n \implies \mathcal{N}\left( \mathbf{0}, \frac{2}{3(1-2a)(2-a)d}I_d \right).
\end{align*}
\end{cor}
\begin{proof}
We observe that 
\begin{align*}
\frac{1}{\sqrt{n}} G_n = \int_0^1 \frac{S_{\lfloor nt \rfloor}}{\sqrt{n}}\mathrm{d}t,
\end{align*}
hence $n^{-1/2} G_n$ is a continuous functional of $n^{-1/2} S_{\lfloor nt \rfloor}$ in $D([0,1])$. 
We apply Theorem \ref{thm:cltMERW} and obtain as $n$ tends to infinity
\begin{align*}
\frac{1}{\sqrt{n}}G_n = \int_0^1 \frac{S_{ \lfloor nt \rfloor}}{\sqrt{n}}\mathrm{d}t \implies \int_0^1 W_t \mathrm{d}t.
\end{align*}
As $W$ is a continuous $\mathbb{R}^d$-valued centered Gaussian process started from zero, so is $\int_0^1 W_t \mathrm{d}t$ and we obtain that its covariance structure is given by
\begin{align*}
\EE \left[ \left( \int_0^1 W_t \mathrm{d}t \right) \left( \int_0^1 W_s \mathrm{d}s \right)'  \right] 
&=2 \int_0^1\int_0^t \EE[W_tW_s'] \mathrm{d}s \mathrm{d}t  \\
&= \frac{2(2d-1)}{d(1+2d-4dp)}\int_0^1 \int_0^t s \left( \frac{t}{s}\right)^{ \frac{2dp-1}{2d-1}}I_d \mathrm{d}s\mathrm{d}t  \\ 
&= \frac{2}{3(1-2a)(2-a)d}I_d.
\end{align*}
This concludes the proof.
\end{proof}
\subsection{The critical regime $(p=p_c^d)$}
In the borderline case of $p=p_c^d$ as given in (\ref{eq:criticalmemoryparameter}) we provide another functional limit theorem for the MERW when properly normalised. 
\begin{thm} \label{thm:MERWcritical} Let $p=p_c^d$. Then, for $n$ tending to infinity, we have the distributional convergence in $D([0, \infty))$:
\begin{align*}
\left( \frac{S_{\lfloor n^t\rfloor}}{\sqrt{\log n}n^{t/2}}, t \geq 0 \right) \implies \frac{1}{\sqrt{d}}(B_t, t \geq 0),
\end{align*}
where $B=(B_t, t \geq 0)$ is a standard $d$-dimensional Brownian motion. 
\end{thm}
Again, for $t=1$, we rediscover the asymptotic normality provided in Theorem 3.6 in Bercu and Laulin \cite{berculaulin}. For $d=1$, that is the ERW, we arrive again at Theorem 2 established by Baur and Bertoin \cite{bertoinbaur}.
\begin{proof}[Proof of Theorem \ref{thm:MERWcritical}]
In this critical regime Theorem 3.31(ii) of Janson \cite{janson} applies, it entails that
\begin{align*}
\left( ( \log n)^{-1/2} n^{-t/2} \left( X_{\lfloor n^t \rfloor} - n^t \lambda_1 v_1 \right), t \geq 0 \right)
\end{align*}
converges in distribution in $D([0, \infty))$ towards a continuous $\mathbb{R}^{2d}$-valued centered Gaussian process $V=(V_t, t \geq 0 )$ started from zero. The covariance structure of $V$ is given by Display (3.27) in Janson \cite{janson}, in our case (see Appendix \ref{Appendix} for more details) it simplifies to
\begin{align} \label{proofstep:covariancealmostdone}
\EE(V_sV_t') = \frac{s}{4d^2}  \begin{pmatrix}
2d-1 & -1 & \dots & -1  \\
-1 & \ddots  & \ddots  & \vdots  \\
\vdots  & \ddots & \ddots & -1  \\
-1  & \dots & -1 & 2d-1
\end{pmatrix}, \quad 0 < s \leq t.
\end{align}
As before we have $\lambda_1=1$ and $v_1= \frac{1}{2d}(1, \dots ,1)' \in \mathbb{R}^{2d}$, thus we conclude by (\ref{eq:connectionMERWUrn}) that we have the distributional convergence in $D([0, \infty))$:
\begin{align*}
\left( \frac{S_{\lfloor n^t \rfloor}}{\sqrt{ \log n } n^{t/2}}, t\geq 0 \right) \implies (W_t, t \geq 0),
\end{align*}
where $W=(W_t,t \geq 0)$ is given by $W_t=(V_t^1-V_t^2)e_1' + \dots + (V_t^{2d-1}-V_t^{2d})e_d'$ almost surely and $V^i$ denotes again the $i$-th component of the process $V$ for all $i=1, \dots , 2d$. As $W$ is a linear combination of the components of $V$, it follows that $W$ is a continuous $\mathbb{R}^d$-valued centered Gaussian process started from zero and thus the law of $W$ is completely determined by its covariance structure. Straightforward calculations using (\ref{proofstep:covariancealmostdone}) yield that
\begin{align*}
\EE(W_sW_t')= \frac{1}{d}s I_d, \quad 0 < s \leq t,
\end{align*}
which lets us conclude that $\sqrt{d}W= B=(B_t, t \geq 0)$ is a standard $d$-dimensional Brownian motion. 
\end{proof}
\subsection{The superdiffusive regime $(p_c^d<p<1)$}
In this regime we provide an alternative proof of Theorem 3.7 in Bercu and Laulin \cite{berculaulin}.
\begin{thm}[Bercu and Laulin \cite{berculaulin}, Theorem 3.7] \label{thm:superdiffusiveMERW} Suppose that $p_c^d < p < 1$ and let us denote $\alpha =(2dp-1)/(2d-1)$. Then, for $n$ tending to infinity, we have the almost sure convergence 
\begin{align*}
\left( \frac{S_{ \lfloor tn \rfloor}}{n^\alpha}, t \geq 0 \right) \to (t^\alpha Y, t \geq 0),
\end{align*}
where $Y$ is some $\mathbb{R}^d$-valued random variable different from zero. 
\end{thm}
In stark contrast to the \textit{diffusive} and \textit{critical regime} discussed in the previous two sections, the distribution of $Y$ does in fact depend on the law of the initial step of the MERW. We observe for example that for $p=1$, the behaviour of the MERW becomes deterministic, indeed we have $S_{\lfloor tn \rfloor}= \lfloor tn\rfloor S_1$ almost surely and the assertion of Theorem \ref{thm:superdiffusiveMERW} trivially holds with $Y$ having the same distribution as $S_1$. In this direction, see also the remarks above Theorem 3.9 in Janson \cite{janson}. Notably, Bercu in \cite{bercu} provided the first rigorous mathematical proof showing that in the case $d=1$ the random variable $Y$ does not follow a Gaussian distribution. 
\begin{proof}[Proof of Theorem \ref{thm:superdiffusiveMERW}]
We observe that, in the notation of Theorem 3.24 of Janson \cite{janson}, we have $\Lambda_{III}' = \{ (2dp-1)/(2d-1)\}=\{ \lambda_2\}.$ We are therefore in the setting of the last part of the cited theorem and get that 
\begin{align*}
\left( n^{- \alpha} \left( X_{ \lfloor tn \rfloor}-tn \lambda_1 v_1 \right), t \geq 0 \right)
\end{align*}
converges almost surely to $\left( t^\alpha \hat{W}_t, t \geq 0 \right),$ where $\hat{W}=\left( \hat{W}^1, \dots , \hat{W}^{2d} \right)$ is some non-zero random vector lying in the eigenspace $E_{\lambda_2}$ of the mean replacement matrix $A$ as displayed in (\ref{mat:meanreplacementmat}), that is  
\begin{align*}
\hat{W} \in \{ v \in \mathbb{R}^{2d} : v = \lambda v_2 \text{ for some } \lambda \in \mathbb{R} \setminus \{0\} \},
\end{align*}
where $v_2$ is the corresponding eigenvector to the eigenvalue $\lambda_2$. Since $\lambda_1=1$ and $v_1= \frac{1}{2d}(1, \dots , 1)' \in \mathbb{R}^{2d}$ we appeal to the connection established in (\ref{eq:connectionMERWUrn}) and conclude the proof as we note that $$Y=\left( \hat{W}^1- \hat{W}^2\right)e_1' + \dots + \left( \hat{W}^{2d-1} - \hat{W}^{2d} \right) e_d'$$ almost surely.
\end{proof}
\section*{Acknowledgements}
I would like to express my gratitude to both of my supervisors, Jean Bertoin and Erich Baur, for their patience and valuable input while I was writing this article. Further, I want to thank two anonymous referees for their careful reading of an earlier version of this manuscript and their many insightful comments and suggestions to help improve this article.
\appendix
\section{Appendix: Some technical calculations} \label{Appendix}
In this section we provide the interested reader for convenience with technical calculations that have been used in Section \ref{sec:four} during the proofs of the main results. 

We recall by (\ref{mat:meanreplacementmat}), see also Display (2.1) in Janson \cite{janson}, that the \textit{mean replacement matrix} of the urn process introduced in Section \ref{sec:three} is given by 
\begin{align*}
A = \begin{pmatrix}
p & \frac{1-p}{2d-1} & \dots & \dots & \frac{1-p}{2d-1}\\
\frac{1-p}{2d-1} & \ddots & \ddots &  & \vdots \\ 
\vdots & \ddots&\ddots & \ddots & \vdots \\
 \vdots &  & \ddots  & \ddots & \frac{1-p}{2d-1} \\
 \frac{1-p}{2d-1} & \dots & \dots&  \frac{1-p}{2d-1} & p
\end{pmatrix}.
\end{align*}
Recall that $A$ has eigenvalues $\lambda_1=1$ and $\lambda_2=(2dp-1)/(2d-1)$, where $\lambda_1$ is of multiplicity one and $\lambda_2$ is of multiplicity $2d-1$. The right and left eigenvectors corresponding to the largest eigenvalue $\lambda_1$ are given by $v_1 = \frac{1}{2d}(1, \dots ,1 )' \in \mathbb{R}^{2d}$ and $u_1=(1, \dots , 1) \in \mathbb{R}^{2d}$ respectively. We have chosen $v_1$ such that its $L^1$-norm is one and $u_1v_1=1$. 

We will use the Jordan decomposition of the matrix A in the following form (see for example \cite{nomizu}, 1979, Theorem 7.6). There exists a decomposition of the complex space $\mathbb{C}^{2d}$ as a direct sum $\oplus E_{\lambda}$ of generalised eigenspaces $E_{\lambda}$, where $\lambda = \lambda_1,\lambda_2$. Equivalently, there exist (orthogonal) projections $P_{\lambda_1}, P_{\lambda_2}$ that commute with $A$ and satisfy \begin{align} 
P_{\lambda_1} + P_{\lambda_2}&=I_{2d}, \label{eq:DecompositionIdentityMatrix}\\
AP_{\lambda}  & = \lambda P_{\lambda}, \label{eq:Commuting}
\end{align}
 where $I_{2d}$ denotes the $2d \times 2d$ identity matrix. We have
\begin{align*}
P_{\lambda_1}= v_1u_1' = \frac{1}{2d}J_{2d},
\end{align*} 
where $J_{2d}$ denotes the $2d \times 2d$ matrix saturated by ones and then, by (\ref{eq:DecompositionIdentityMatrix}), also
\begin{align} \label{proj:Projection2}
P_{\lambda_2} = \frac{1}{2d}\begin{pmatrix}
2d-1 & -1 & \dots & -1  \\
-1 & \ddots  & \ddots  & \vdots  \\
\vdots  & \ddots & \ddots & -1  \\
-1  & \dots & -1 & 2d-1
\end{pmatrix}.
\end{align} 
Key quantities that govern the asymptotic behaviour of the Pólya urn are given by the following $2d \times 2d$ matrices, specified under (2.15) respectively (2.16) in Janson \cite{janson}:
\begin{align}
\Sigma_I &:= \frac{1}{2d} \int_0^\infty P_{\lambda_2} e^{2sA} \label{eq:SigmaI} P_{\lambda_2} e^{- \lambda_1s} \mathrm{d}s, \\
\Sigma_{II} &:= \frac{1}{2d} P_{\lambda_2}. \label{eq:SigmaII}
\end{align}
By the definition of the matrix exponential and (\ref{eq:Commuting}) one finds that for all $s \in \mathbb{R}$ $$P_{\lambda_2} e^{2sA} = e^{2s \lambda_2} P_{\lambda_2}$$ and, since $P_{\lambda_2}^2= P_{\lambda_2}$, (\ref{eq:SigmaI}) simplifies to 
\begin{align*}
\Sigma_I = \frac{1}{2d} \int_0^\infty e^{(2 \lambda_2-1)s} P_{\lambda_2} \mathrm{d}s.
\end{align*}
We notice that the integral above converges if and only if $p<p_c^d$ and in that case we obtain 
\begin{align*}
\Sigma_I = \frac{2d-1}{4d^2(1+2d-4dp)} P_{\lambda_2}
\end{align*}
and, with an appeal to (\ref{proj:Projection2}), we arrive at (\ref{eq:sigmaDiffusive}) used during the proof of Theorem \ref{thm:cltMERW}. Further, we have for all $s \in \mathbb{R}$
\begin{align*}
e^{sA} = e^{ \lambda_1s} P_{\lambda_1} + e^{\lambda_2s} P_{\lambda_2}
\end{align*}
and then, again by the orthogonality of the projections $P_{\lambda_1}, P_{\lambda_2}$, we obtain for all $0<s \leq t$ 
\begin{align*}
s \Sigma_I e^{ \log (t/s)A} &=  s \frac{2d-1}{2d(1+2d-4dp)}P_{\lambda_2} \left[ \left( \frac{t}{s}\right)^{\lambda_1}P_{\lambda_1} + \left( \frac{t}{s}\right)^{\lambda_2}P_{\lambda_2} \right] \\
&= \frac{s(2d-1)}{2d(1+2d-4dp)} \left( \frac{t}{s}\right)^{\lambda_2} P_{\lambda_2}
\end{align*}
and we arrive at (\ref{eq:covariancealmostdone}).

Finally, thanks to Display (3.27) in Janson \cite{janson} we know that the covariance structure of the $\mathbb{R}^{2d}$-valued centered Gaussian process $V$, used during the proof of Theorem \ref{thm:MERWcritical}, is given by 
\begin{align*}
\EE(V_s V_t')= s \Sigma_{II}, \qquad 0 < s \leq t,
\end{align*}
and we arrive at (\ref{proofstep:covariancealmostdone}).
\bibliographystyle{abbrvurl}
\bibliography{bibliography}
\end{document}